\documentclass[12pt]{article}
\def\mydate{19 August 2016}
\usepackage{amsmath, amssymb, amsfonts, theorem,graphicx}
\usepackage[dvips]{color}
\usepackage[dvipsnames]{xcolor}

\newtheorem{proposition}{proposition}[section]

\newtheorem{theorem}[proposition]{Theorem}
\newtheorem{claim}[proposition]{Claim}

\newtheorem{prop}[proposition]{Proposition}
\newtheorem{lem}[proposition]{Lemma}

\newtheorem{thm}[proposition]{Theorem}

\newcommand{\qed}{\hfill$\square$\bigskip}

{\theorembodyfont{\rmfamily}
\newtheorem{definition}[proposition]{Definition}

}

\newcommand{\mcal}{\mathcal}

\newcommand{\C}{\mcal{C}}

\renewcommand{\P}{\mcal{P}}

\def\aa#1{{\color{magenta}#1}}
\def\aa#1{{#1}}
\def\xx#1{{\color{ForestGreen}#1}}
\def\xx#1{{\color{magenta}#1}}
\def\xx#1{{#1}}
\def\REM#1{\footnote{#1}}
\def\REM#1{}
\def\REMA#1{\footnote{#1}}
\def\REMA#1{}

\begin{document}
\font\smallrm=cmr8




\phantom{a}\vskip .25in
\centerline{{\large \bf  FIVE-LIST-COLORING GRAPHS ON SURFACES III.}}
\smallskip
\centerline{{\large\bf ONE LIST OF SIZE ONE AND ONE LIST OF SIZE TWO}}
\vskip.4in
\centerline{{\bf Luke Postle}%
\footnote{\texttt{lpostle@uwaterloo.ca}. Partially supported by NSERC under Discovery Grant No. 2014-06162.}} 
\smallskip
\centerline{Department of Combinatorics and Optimization}
\centerline{University of Waterloo}
\centerline{Waterloo, ON}
\centerline{Canada N2L 3G1}
\medskip
\centerline{and}

\medskip
\centerline{{\bf Robin Thomas}%
\footnote{\texttt{thomas@math.gatech.edu}. Partially supported by NSF under
Grant No.~DMS-1202640.}}
\smallskip
\centerline{School of Mathematics}
\centerline{Georgia Institute of Technology}
\centerline{Atlanta, Georgia  30332-0160, USA}

\vskip 1in \centerline{\bf ABSTRACT}
\bigskip

{
\noindent
Let $G$ be a plane graph with outer cycle $C$ and let $(L(v):v\in V(G))$ be a family of non-empty sets.
By an $L$-coloring 
of $G$ we mean a (proper) coloring $\phi$ of $G$  such that $\phi(v)\in L(v)$ for every vertex $v$ of $G$.
Thomassen proved that if $v_1,v_2\in V(C)$ are adjacent, $L(v_1)\ne L(v_2)$,
$|L(v)|\ge3$ for every $v\in V(C)-\{v_1,v_2\}$ and $|L(v)|\ge5$ for every $v\in V(G)-V(C)$,
then $G$ has an $L$-coloring.
What happens when $v_1$ and $v_2$ are not adjacent?
Then an $L$-coloring need not exist, but
in the first paper of this series we have shown that it exists if $|L(v_1)|,|L(v_2)|\ge2$.
Here we characterize when an $L$-coloring exists if $|L(v_1)|\ge1$ and $|L(v_2)|\ge2$.

This result is a lemma toward a more general theorem along the same lines, which
we will use to prove that minimally non-$L$-colorable planar graphs with two precolored cycles
of bounded length are of bounded size.
The latter result has  a number of applications which we pursue elsewhere.
}

\vfill \baselineskip 11pt \noindent January 22 2014. Revised \mydate.

\vfil\eject
\baselineskip 18pt

\section{Introduction}


\label{sec:intro}
All {\em graphs} in this paper are finite and simple; that is, they have no loops or parallel edges. 
{\em Paths} and {\em cycles} have no repeated vertices 
or edges. If $G$ is a graph and $L=(L(v):v\in V(G))$ is a family of non-empty sets, then we say that $L$ is a {\em list assignment for $G$}.
It is a {\em $k$-list-assignment}, if $|L(v)|\ge k$ for every vertex $v\in V(G)$.
An $L$-coloring  of $G$ is a (proper) coloring $\phi$ of $G$  such that $\phi(v)\in L(v)$
for every vertex $v$ of $G$.
We say that a graph $G$ is \emph{$k$-choosable}, also called \emph{$k$-list-colorable}, if for every $k$-list-assignment $L$ for $G$, $G$ has an $L$-coloring. 

%
%

One notable difference between list coloring and ordinary coloring is that the 
Four Color Theorem~\cite{4CT1, 4CT2} does not generalize to list-coloring. 
Indeed, Voigt~\cite{Voigt} constructed a planar graph that is not $4$-choosable. 
On the other hand Thomassen~\cite{ThomPlanar} proved that every planar graph is $5$-choosable.
His proof is remarkably short and beautiful. For the sake of the inductive argument he proves the following
stronger statement.

%

\begin{theorem}\label{Thom}
If $G$ is a plane graph with outer cycle $C$ and $P=p_1p_2$ is a path of length one in $C$ and $L$ is a list assignment with $|L(v)|\ge 5$ for all $v\in V(G)- V(C)$, $|L(v)|\ge 3$ for all $v\in V(C)- V(P)$, and $|L(p_1)|,|L(p_2)|\ge1$ with $L(p_1)\ne L(p_2)$, then $G$ is $L$-colorable.
\end{theorem}

What if $p_1$ and $p_2$ are not adjacent?
In that case $G$ need not be $L$-colorable, but it is possible to characterize instances
when it is not. In fact, we are able to extract useful information even when more vertices are pre-colored,
but it will take some effort. We began this line of research
in our previous paper~\cite{PT1}, where we proved a generalization of Theorem~\ref{Thom} conjectured by Hutchinson~\cite{HutchOuterplanar}, who proved the result for outerplanar graphs. 

\begin{theorem}\label{TwoTwos2}
If $G$ is a plane graph with outer cycle $C$ and $p_1,p_2\in V(C)$ and $L$ is a list assignment with $|L(v)|\ge 5$ for all $v\in V(G)- V(C)$, $|L(v)|\ge 3$ for all $v\in V(C)-\{p_1,p_2\}$, and $|L(p_1)|,|L(p_2)|\ge2$, then $G$ is $L$-colorable.
\end{theorem}

The main result of this paper is to characterize when an $L$-coloring exists, if in
Theorem~\ref{TwoTwos2} we only assume that $|L(p_1)|\ge1$. In order to state the theorem
we need to define a family of obstructions.

Let $G$ be a connected plane graph, and let $u,v,w$ be distinct vertices of $G$ incident 
with the outer face of $G$,
let $u$ be adjacent to $v$,  let the edge $uv$ be incident with the outer face of $G$
and let $L$ be a list assignment for $G$.
We say that the pair $(G,L)$ is a {\em coloring harmonica from $uv$ to $w$} if either
\begin{itemize}
\item  $G$ is a triangle with vertex set $\{u,v,w\}$, $L(u)=L(v)=L(w)$ and $|L(u)|=2$, or
\item there exists a vertex $z\in V(G)$ incident with the outer face of $G$
such that $uvz$ is a triangle in $G$,
$L(u)=L(v)\subseteq L(z)$,  $|L(u)|=|L(v)|=2$, $|L(z)|=3$, and $(G',L')$ is
a coloring harmonica from $z$ to $w$, where $G'$ is obtained from $G$ by deleting
 one
or both of the vertices $u,v$, and $L'$ satisfies  $L'(z)=L(z)-L(u)$ and $L'(x)=L(x)$ 
for every $x\in V(G')-\{z\}$.
\end{itemize}

\noindent 
We say that the pair $(G,L)$ is a {\em coloring harmonica from $u$ to $w$} if 
\begin{itemize}
\item there exist  vertices $x,y\in V(G)$  incident with the outer face of $G$ such that $uxy$ is a triangle in $G$,
$|L(u)|=1$, $L(x)- L(u)=L(y)-L(u)$,  $|L(x)-L(u)|=2$,  and $(G',L')$ is
a coloring harmonica from $xy$ to $w$, where $G':=G\setminus u$,
   $L'(x)=L(x)-L(u)$,  $L'(y)=L(y)-L(u)$ and $L'(z)=L(z)$ for every $z\in V(G')-\{x,y\}$.
\end{itemize}

\noindent 
See Figure~\ref{fig:harmon}.
We say that the pair $(G,L)$ is a {\em coloring harmonica} if it is either a
coloring harmonica from $uv$ to $w$ or a coloring harmonica from $u$ to $w$, where
 $u,v,w$ are as specified earlier.
We say that the pair $(G,L)$ {\em contains a  coloring harmonica $(G',L')$} 
if $G'$ is a subgraph of $G$ and $L'(x)=L(x)$ for every $x\in V(G')$.


\begin{figure}[htb]
 \centering
\includegraphics[scale = .25]{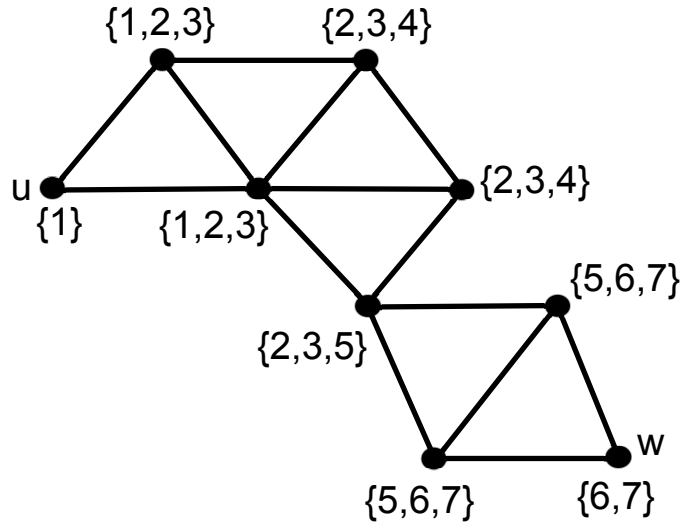}
 \caption{A coloring harmonica from $u$ to $w$.}
\label{fig:harmon}
\end{figure}

We can now state the main result of this paper.
Hutchinson~\cite{HutchOuterplanar} proved it for outerplanar graphs.

\begin{theorem}
\label{thm:harmon}
Let $G$ be a plane graph with outer cycle $C$, let $p_1,p_2\in V(C)$, and let $L$ be a list assignment 
with $|L(v)|\ge 5$ for all $v\in V(G)- V(C)$, $|L(v)|\ge 3$ for all $v\in V(C)- \{p_1,p_2\}$,  $|L(p_1)|\ge1$ and $|L(p_2)|\ge2$. Then $G$ is $L$-colorable if and only if the pair $(G,L)$
does not contain a coloring  harmonica from $p_1$ to $p_2$.
\end{theorem}


\section{Canvases}\label{sec:canvases}

We also recall the definition of the graphs we are working with, first introduced in our previous paper~\cite{PT1}.

\begin{definition}[Canvas]

We say that $(G, S, L)$ is a \emph{canvas} if $G$ is a plane graph, $S$ is a subgraph of the boundary of the outer  face of $G$, and $L$ is a list assignment for some supergraph of $G$ such that $|L(v)|\ge 5$ for all $v\in V(G)- V(C)$, where $C$ is the boundary of the outer  face of $G$, $|L(v)|\ge 3$ for all $v\in V(G)- V(S)$, and there exists a proper $L$-coloring of $S$.
%
\REMA{Deleted: We say the canvas $(G,S,L)$ is \emph{$L$-critical} if there does not an exist an $L$-coloring of $G$ but for every edge $e\not\in E(S)$, there exists an $L$-coloring of $G\setminus e$, and for every vertex $v\not\in V(S)$, there exists an $L$-coloring of $G\setminus v$.}
\end{definition}


We should remark that we allow $L$ to be a list assignment of some supergraph of $G$ merely for convenience when passing to subgraphs. Given this definition of a canvas, we can state an equivalent but slightly more general version of Theorem~\ref{Thom} as follows.

\begin{thm}\label{ThomPath2}
If $(G,S,L)$ is a canvas and $S$ is path of length one, then $G$ is $L$-colorable. 
\end{thm}

We can also restate Theorem~\ref{TwoTwos2} in these terms.

\begin{thm}[Two with List of Size Two Theorem]\label{TwoTwos} 
If $(G,S,L)$ is a canvas with  $\allowbreak V(S)=\{v_1,v_2\}$ and $|L(v_1)|,|L(v_2)|\ge 2$, then $G$ is $L$-colorable.
\end{thm}

It should be noted that Thomassen~\cite{ThomWheels} characterized the canvases $(G,S,L)$ where $S$ is a path of length two and $G$ is not $L$-colorable. For our main theorem, we do not need this full characterization. However, we do need the following lemma that can be found in~\cite[Lemma~1]{ThomWheels}.
A {\em chord} of a cycle in a graph $G$ is a subgraph of $G$ consisting of two vertices that belong to the cycle and an edge joining them that does not belong to the cycle.

\begin{lem} \label{WheelUniqueColor}
Let $T=(G,S,L)$ be a canvas such that $S$ is an induced path of length two. If there does not exist a chord of the boundary of the outer face of $G$, then there exists at most one proper $L$-coloring of $S$ that does not extend to an $L$-coloring of $G$.%
\REMA{Deleted definition of essential vertex and essential chord}
\end{lem}

We also need a notion of containment for canvases as follows.

\begin{definition}
A canvas $T=(G,S,L)$ \emph{contains} a canvas $T'=(G',S',L')$ if $G'$ is a subgraph of $G$, $S=S'$ and the restrictions of $L$ and $L'$ to $G'$ are equal.
\end{definition}

\section{Governments and Reductions}

In this section, we will develop notation and definitions for characterizing how the colorings of $P$ in Theorem~\ref{Thom} extend to colorings of other paths of length one on the boundary of the outer walk. We introduce a notion called a government to describe sets of colorings that come in two types which we call dictatorships and democracies. Our main theorem will show that a government extends to at least two governments unless a very specific structure occurs. 

\subsection {Coloring Extensions}

\begin{definition}
Suppose $T=(G,P,L)$ is a canvas where $P=p_1p_2$ is a path of length one in the boundary $C$ of the outer face of $G$. Suppose we are given a collection $\C$ of $L$-colorings of $P$. Let $P'$ be an edge of $G$ with both ends in $C$. We let $\Phi_T(P',\C)$ denote the collection of proper $L$-colorings of $P'$ that can be extended to a proper $L$-coloring $\phi$ of $G$ such that $\phi$ restricted to $P$ is an $L$-coloring in $\C$. We will drop the subscript $T$ when the canvas is clear from context.
\end{definition}

We may now restate Theorem~\ref{Thom} in these terms. 
\begin{thm}\label{Thom2} 
Let $T=(G,P,L)$ be a canvas with $P$ a path of length one. If $\C$ is a non-empty collection of proper $L$-colorings of $P$, 
and $P'$ is an edge of $G$ with both ends in $C$, then $\Phi(P',\C)$ is nonempty.
\end{thm}

Note the following easy proposition.
\begin{prop} \label{PhiChord}
Let $T,P,P'$ be as in Theorem~\ref{Thom2}. If $U=u_1u_2$ is a chord of $C$ separating $P$ from $P'$, then 
$$\Phi(P', \Phi(U,\C) )=\Phi(P',\C).$$
\end{prop}

\subsection{Governments}

To explain the structure of extending larger sets of colorings, we focus on two special sets of colorings, defined as follows.

\begin{definition}[Government]

Let $\C=\{\phi_1, \phi_2, \ldots, \phi_k\}$, $k\ge 2$, be a collection of disjoint proper colorings of a path $P=p_1p_2$ of length one. For $p\in P$, let $\C(p)$ denote the set $\{\phi(p)|\phi\in \C\}$.

We say $\C$ is a \emph{dictatorship} if there exists $i\in \{1,2\}$ such that $\phi_j(p_i)$ is the same for all $1\le j\le k$, in which case, we say $p_i$ is the dictator of $\C$. We say $\C$ is a \emph{democracy} if $k=2$ and $\phi_1(p_1)=\phi_2(p_2)$ and $\phi_2(p_1)=\phi_1(p_2)$. We say $\C$ is a \emph{government} if $\C$ is a dictatorship or a democracy.
\end{definition}

We also need a generalized form of government as follows.

\begin{definition}
Let $\C$ be a collection of disjoint proper colorings of a path $P=p_1p_2$ of length one. We say $\C$ is a \emph{confederacy} if $\C$ is not a government and yet $\C$ is the union of two governments. 
\end{definition}

\subsection{Reductions}

Thomassen found a useful reduction in his proof of $5$-choosability. We will need a generalization of that reduction as follows.

\begin{definition}[Democratic Reduction]
Let $T=(G,S,L)$ be a canvas and $L_0$ be a set of two colors. Let $C$ be the boundary of the outer face of $G$. Suppose that $P=p_1\ldots p_k$ is an induced path in $C$ such that, for every vertex $v$ in $P$, $v$ is not the end of a chord of $C$ or a cutvertex of $C$,
\aa{$V(C)\ne V(P)$,} and $L_0\subseteq L(v)$. If $k\ge 2$, let $x$ be the vertex of $C$ adjacent to $p_1$ other than $p_2$ and $y$ be the vertex of $C$ adjacent to $p_k$ other than $p_{k-1}$. If $k=1$, let $x$ and $y$ be the two neighbors of $p_1$ on $C$. We assume that $L(x)- L_0\ne \emptyset$. 

We define the \emph{democratic reduction} of $P$ in $T$ with respect to $L_0$ and centered at $x$, denoted as $T(P,L_0,x)$, as $(G\setminus V(P),S',L')$ where $L'(w)=L(w)- L_0$ if either $w=x$, or $w\ne y$ is a neighbor of a vertex in $P$, and $L'(w)=L(w)$ otherwise; and $S'=S\setminus V(P)$ if $|L'(x)|\ge 3$ and otherwise let $S'$ be obtained from $S\setminus V(P)$ by adding $x$ as an isolated vertex.
\end{definition}

\REM{Note: CHECK Later if want edges from $x$ to $S\setminus V(P)$. CHECK Later if we ever use $V(P)=V(C)$.}

\begin{prop}\label{DemRed}
Let $T(P,L_0,x)=(G',S',L')$ be a democratic reduction of a path $P$ in a canvas $T=(G,S,L)$ with respect to $L_0$ and centered at $x$. The following statements hold:

\begin{enumerate}
\item $T(P,L_0, x)$ is a canvas.
\item If $\phi$ is an $L'$-coloring of $G'$, then $\phi$ can be extended to an $L$-coloring of $G$.
\end{enumerate}
\end{prop}
\begin{proof}
Let $C$ be the boundary of the outer face of $G$. If $v\in V(G')$ such that $|L'(v)|<5$, then either $v\in C$ or $v$ is adjacent to a vertex of $w$ in $P$. In either case, $v$ is in the boundary of the outer face of $G'$. Note that if $v\in V(G')$ such that $L'(v)\ne L(v)$, then either $v=x$ or $v\not\in C$. In the latter case, $|L(v)|=5$ and hence $|L'(v)|\ge 3$. Thus, if $v\in V(G')$ such that $|L'(v)|<3$, then $v\in V(S)\cup \{x\}$. Recall that by definition, $V(S')=V(S)\cup \{x\}$ if $|L'(x)|<3$ and $V(S')=V(S)$ otherwise. In either case, it follows that if $v\in V(G')$ such that $|L'(v)| < 3$, then $v\in V(S')$. This proves (1).

Let $\phi$ be an $L'$-coloring of $G'$. Let $P=p_1\ldots p_k$ where $p_1$ is adjacent to $x$. If $k=1$, let $y$ be the neighbor of $p_1$ in $C$ other than $x$. If $k\ge 2$, let $y$ be the neighbor of $p_k$ in $C$ other than $p_{k-1}$. Let $\phi(p_k)\in L_0- \{\phi(y)\}$. For all $i$ with $1\le i\le k-1$, let $\phi(p_i)\in L_0-\{\phi(p_{i+1})\}$. Now $\phi$ is an $L$-coloring of $G$. This proves (2).
%
\qed
\end{proof}

We note that Thomassen's reduction corresponds to a democratic reduction where $|V(P)|=1$, $x\in V(S)$ and $|L(x)|=1$.

\section{Harmonicas}\label{sec:harmonicas}

In this section, we rework the definition of coloring harmonica to an object involving governments which we call a harmonica. We then prove a stronger version of our main theorem which shows that harmonicas are the only obstacle to extending a government to a confederacy. We will then show that this implies that coloring harmonicas are the only obstruction to generalizing Theorem~\ref{TwoTwos} to the case of one vertex with a list of size one and one with a list of size two. That is, we finally prove Theorem~\ref{thm:harmon}.

\begin{definition}[Harmonica]
Let $T=(G,P,L)$ be a canvas where $P$ is path of length one. Let $\C$ be a government for $P$ and let $P'$ be another
(not necessarily distinct) path of length one incident with the outer face of $G$. We say $T$ is a \emph{harmonica} from $P$ to $P'$ with government $\C$ if 

\begin{itemize}
\item $G=P=P'$, or
\item $\C$ is a dictatorship, $G=P\cup P'$, $V(P)\cap V(P') = \{z\}$ where $z$ is the dictator of $\C$, or
\item $\C$ is a dictatorship and there exists a triangle $zu_1u_2$ where $z\in V(P)$ is the dictator of $\C$ in color $c$,
\aa{for $i=1,2$ we have $L(u_i)=L_0\cup\{c\}$ if $u_i\not\in V(P)$ and $\C(u_i)=L_0
$ otherwise,}
where $|L_0|=2$ and the canvas $(G\setminus (V(P)- V(U)), U,L)$ is a harmonica from $U=u_1u_2$ to $P'$ with democracy $\C'$ where $\C'(u_1)=\C'(u_2)=L_0$, or
\item $\C$ is a democracy and there exists $z\sim p_1,p_2$ where $P=p_1p_2$ such that $L(z)=L_0\cup \{c\}$ where $L_0=\C(p_1)=\C(p_2)
$ and there exists $i\in\{1,2\}$ such that the canvas $(G\setminus p_i, U, L)$ is a harmonica with dictatorship $\C'=\{\phi_1,\phi_2\}$, where $U=zp_{3-i}$ and $\phi_1(z)=\phi_2(z)=c$ and $\{\phi_1(p_{3-i}),\phi_2(p_{3-i})\}=L_0$.
\end{itemize}

Note that $\Phi_T(P',\C)$ is a government.%
\REM{Deleted: We say a harmonica is \emph{even} if $\C$ and $\C'$ are both dictatorships or both democracies and \emph{odd} otherwise.}
We remark that the notion of harmonica is closely related to the notion of coloring harmonica, introduced earlier.
Lemma~\ref{lem:harmonicas} clarifies the relation between the two.
\end{definition}

\xx{%
We need the following easy lemma, whose proof we omit.}

\xx{%
\begin{lem} 
\label{L=3}
Let $T=(G,P,L)$ be a harmonica from $P$ to $P'$ with government $\C$, and let $v\in V(G)-V(P)$ be such that 
if $v\in V(P')$, then $v$ has degree at least two. Then $|L(v)|=3$.
\end{lem}
}

The following is our main result.

\begin{thm}\label{Harmonica2}
Let $T=(G,P,L)$ be a canvas and $P,P'$ be paths of length one in the boundary of the outer face of $G$. Given a collection $\C$ of proper colorings of $P$ such that $\C$ is a government or a confederacy, then $\Phi(P',\C)$ contains a government, and furthermore, either 

\begin{itemize}
\item $\Phi(P',\C)$ contains a confederacy, or, 
\item $\C$ is a government and $T$ contains a harmonica from $P$ to $P'$ with government $\C$.
\end{itemize}
\end{thm}

\begin{proof}
Suppose that $T=(G,P,L)$ is a counterexample with $|V(G)|$ minimized and, subject to that, $\C$ is a government if possible. Let $C$ be the boundary of the outer face of $G$.

\begin{claim}
$G$ is connected.
\end{claim}
\begin{proof}
Suppose not. Let $G_1$ be the component of $G$ containing $P$. First suppose that $G_1$ contains $P'$ and let $G_2$ be a component of $G$ other than $G_1$. Let $T'=(G\setminus V(G_2),P,L)$. By the minimality of $G$, $\Phi_{T'}(P',\C)$ contains a government and hence $\Phi_T(P',\C)$ contains a government \aa{by Theorem~\ref{Thom}}. Furthermore, either $\Phi_{T'}(P',\C)$ contains a confederacy, a contradiction as then $\Phi_T(P',\C)$ contains a confederacy  \aa{by Theorem~\ref{Thom}}, or $T'$ contains a harmonica  from $P$ to $P'$ with government $\C$, \aa{in which case so does $T$,} a contradiction. 

So we may assume that $G_1$ does not contain $P'$ and let $G_2$ be the component of $G$ containing $P'$. It follows from Theorem~\ref{ThomPath2} that every $L$-coloring of $P\cup P'$ extends to an $L$-coloring of $G$. In particular, let $P'=p_1'p_2'$ and $c_1\in L(p_1'), c_2\in L(p_2')$. If we let $\C_1=\{\phi_1,\phi_2\}$ where $\phi_1(p_1')=\phi_2(p_1')=c_1$ and $\{\phi_1(p_2'),\phi_2(p_2')\}$ be a subset of $L(p_2')- \{c_1\}$ of size two. Similarly let $\C_2=\{\psi_1,\psi_2\}$ where $\psi_1(p_2')=\psi_2(p_2')=c_2$ and $\{\psi_1(p_1'),\psi_2(p_1')\}$ be a subset of $L(p_1')- \{c_2\}$ of size two. Thus $\C_1$ is a dictatorship with dictator $p_1'$ and $\C_2$ is a dictatorship with dictator $p_2'$. Hence $\C_1\cup \C_2$ is a confederacy and $\C_1\cup \C_2\subseteq \Phi(\P',\C)$, a contradiction.
\qed\end{proof}

\begin{claim}\label{HarmInternal}
There does not exist a vertex in an open disk bounded by a cycle of length at most four.
\end{claim}
\begin{proof}
Let $C$ be a cycle of length at most four in $G$. Let $\Delta$ be the closed disk bounded by $C$. Let $G_1=G\setminus (\Delta \setminus C)$ and $G_2=G\cap \Delta$. Suppose $G\cap (\Delta \setminus C)\ne\emptyset$. Let $\phi$ be an $L$-coloring $\phi$ of $G_1$. It follows from a theorem of Bohme at al~\cite{Bohme} that $\phi$ can be extended to an $L$-coloring of $G_2$ and hence to an $L$-coloring of $G$. Let $T_1=(G_1,P,L)$. From above, $\Phi_{T_1}(P',\C)\subseteq \Phi_T(P',\C)$. Since $P,P'\subseteq G_1$, it follows from the minimality of $T$ that $\Phi_{T_1}(P',\C)$ contains a government $\C'$. Furthermore, either $\Phi_{T_1}(P',\C)$ contains a confederacy, a contradiction, or, $T_1$ contains a harmonica $T'$ from $P$ to $P'$ with government $\C$, and hence so does $T$, a contradiction.
\qed\end{proof}

\begin{claim}\label{Harm2Conn}
$G$ is $2$-connected.
\end{claim}
\begin{proof}
Suppose not. Then there exists a cutvertex $v$ of $G$.
So suppose $v$ divides $G$ into two graphs $G_1,G_2\ne G$ such that $G_1\cup G_2=G$, $V(G_1)\cap V(G_2)=\{v\}$
and without loss of generality $V(P)\subseteq V(G_1)$. Consider the canvases $T_1=(G_1,P,L)$ and $T_2=(G_2,U',L)$  where $U'=vw$ is an edge of the outer walk of $G_2$ containing $v$. If $V(P')\subseteq V(G_1)$, then by the minimality of $T$, 
\aa{$\Phi_{T_1}(P',\C)$ contains a government, and} hence so does $\Phi_{T}(P',\C)$, and  
either $\Phi_{T_1}(P',\C)$ contains a confederacy, or $T_1$ contains a harmonica $T'$ from $P$ to $P'$ with government $\C$. In the former case, it follows from Theorem~\ref{ThomPath2} that every $L$-coloring of $G_1$ extends to an $L$-coloring of $G$ and hence $\Phi_{T_1}(P',\C)$ contains a confederacy, a contradiction. In the latter case, $T$ also contains $T'$, a contradiction. So we may assume that $V(P')\subseteq V(G_2)$.

Now suppose there exist two $L$-colorings $\phi_1,\phi_2$ of $T_1$ such that  $\phi_1(v)\ne\phi_2(v)$. Then there exists a confederacy $\C'$ for $U'$ (a union of two dictatorships) such that every coloring in $\C'$ extends back to $T_1$. As $T$ is a minimum counterexample, it follows that $\Phi_{T_2}(P',\C')$ has a confederacy. Hence $\Phi_T(P',\C)$ has a confederacy, contradicting that $T$ is a counterexample.

Let $U$ be an edge of the outer walk of $G_1$ containing $v$. Hence, by the previous paragraph  we may assume that $\Phi_{T_1}(U,\C)$ is a dictatorship with dictator $v$. Let $c$ be the color of $v$ in that dictatorship. As $T$ is a minimum counterexample, it follows that $T_1$ contains a harmonica $T_1'=(G_1',P,L)$ from $P$ to $U$. 
\REM{Deleted: Moreover, $T_1'$ is an even harmonica if $\C$ is a dictatorship and odd if $\C$ is a democracy.}%
Let $\C_2=\{\psi_1,\psi_2\}$ where $\psi_1(v)=\psi_2(v)=c$ and $\{\psi_1(w),\psi_2(w)\}$ is a subset of $L(w)- \{c\}$ of size two. Note that $\C_2$ is a dictatorship with dictator $v$. It follows from the minimality of $T$ that $\Phi_{T_2}(P',\C_2)$ contains a government and hence $\Phi_T(P',\C)$ contains a government. Furthermore, either $\Phi_{T_2}(P',\C_2)$ contains a confederacy, a contradiction as then $\Phi_T(P',\C)$ contains a confederacy, or that $T_2$ contains a harmonica $T_2'=(G_2',U,L)$ from $U'$ to $P'$. Let $G'$ be the union of $G_1'$ and $G_2'$ where we delete vertices of $U\setminus V(P)$ that have degree one in $G_1'$ and vertices of $U'\setminus V(P')$ that have degree one in $G_2'$. Then $T'=(G',P,L)$ is a harmonica from $P$ to $P'$ with government $\C$, a contradiction.
\qed\end{proof}

\begin{claim}\label{HarmNoChord}
There does not exist a chord of $C$.
\end{claim}
\begin{proof}
Suppose there exists a chord $U$ of $C$. Now $U$ divides $G$ into graphs $G_1,G_2\ne G$ such that $G_1\cup G_2=G$ and $G_1\cap G_2=U$, where we may assume without loss of generality that $P\subseteq G_1$. Consider the canvases $T_1=(G_1,P,L)$ and $T_2=(G_2,U,L)$. If $V(P')\subseteq V(G_1)$, then by the minimality of $T$, 
\aa{$\Phi_{T_1}(\P',\C)$ contains a government, and} hence so does $\Phi_{T}(\P',\C)$, and
either $\Phi_{T_1}(\P',\C)$ contains a confederacy, or $T_1$ contains a harmonica $T'$ from $P$ to $P'$ with government $\C$. In the former case, it follows from Theorem~\ref{ThomPath2} that every $L$-coloring of $G_1$ extends to an $L$-coloring of $G$ and hence $\Phi_{T}(P',\C)$ contains a confederacy, a contradiction. In the latter case, $T$ also contains $T'$, a contradiction. 

So we may assume that $V(P')\subseteq V(G_2)$.  By the minimality of $T$, $\Phi_{T_1}(U,\C)$ contains a government $\C'$. Furthermore, either $\Phi_{T_1}(U,\C)$ contains a confederacy $\C''$, or, there exists a harmonica $T_1'=(G_1',P,L)$ from $P$ to $U$ with government $\C$. Suppose the former. But then by the minimality of $T$, $\Phi_{T_2}(P',\C'')$ contains a confederacy and hence $\Phi_T(P',\C)$ contains a confederacy by Proposition~\ref{PhiChord}, a contradiction.

So we may suppose the latter.  By the minimality of $T$, $\Phi_{T_2}(P',\C')$ contains a government and hence $\Phi_T(P',\C)$ contains a government. Furthermore, either $\Phi_{T_2}(P',\C')$ contains a confederacy or there exists a harmonica $T_2'=(G_2',U,L)$ from $U$ to $P'$ with government $\C'$. If the former holds, then $\Phi_T(P',\C)$ contains a confederacy by Propoisition~\ref{PhiChord}, a contradiction. So suppose the latter. Let $G'$ be obtained from $G_1'$ and $G_2'$ by deleting vertices of $U\setminus (V(P)\cup V(P'))$ that have degree one in both $G_1$ and $G_2$. Then $T'=(G',P,L)$ is a harmonica from $P$ to $P'$ with government $\C$, a contradiction. 
\qed\end{proof}

\begin{claim}\label{NotEqual}
$P\ne P'$.
\end{claim}
\begin{proof}
Suppose not. Note that every $L$-coloring of $P$ extends to an $L$-coloring of $G$ by Theorem~\ref{ThomPath2} and hence $\Phi(P',C)$ contains $\C$. Thus if $\C$ is a confederacy, $\Phi_T(P',\C)$ contains a confederacy, a contradiction. So we may assume that $\C$ is a government but then $(P,P,L)$ is a harmonica from $P$ to $P'$ with government $\C$, a contradiction.
\qed\end{proof}

\begin{claim}\label{NotTriangle}
$V(P)\cup V(P')$ does not induce a triangle.
\end{claim}
\begin{proof}
Suppose it does. Let $\{z\}=V(P)\cap V(P')$ and let $P=xz$ and $P'=yz$. Thus $x$ is adjacent to $y$. By Claims~\ref{HarmInternal} and~\ref{HarmNoChord}, $V(G)=V(P)\cup V(P')$. 

Let $\C_0\subseteq \C$ be a government of $P$ and let $\phi_1\ne \phi_2\in \C_0$. Note that $\Phi_T(P',\C_0)\subseteq \Phi_T(P',\C)$. If $\C_0$ is a democracy, then for every $c\in L(y)- \{\phi_1(\aa{z}),\phi_2(\aa{z})\}$, $\Phi_T(P',\C_0)$ contains a dictatorship with dictator $y$ in color $c$. Hence if there are two such colors, $\Phi_T(P',\C_0)$ contains a confederacy, a contradiction. 
So  in this case, $|\C_0|=2$, $|L(y)|=3$ and $L(y)=\{c,\phi_1(\aa{z}),\phi_2(\aa{z})\}$ for some $c$, $\Phi_T(P',\C_0)$ contains a dictatorship with dictator $y$ in color $c$\aa{, and hence $T$ contains a harmonica from $P$ to $P'$ with government $\C_0$}.

If $\C_0$ is a dictatorship with dictator $z$ in color $a$, then $\Phi_T(P',\C_0)$ contains a dictatorship with dictator $z$ in color 
$a$\aa{, and $T$ contains a harmonica from $P$ to $P'$ with government $\C_0$}. 

Next we claim that if $\C_0$ is a dictatorship with dictator $x$ in color $b$, then $\Phi_T(P',\C_0)$ contains a confederacy, a contradiction, unless $|\C_0|=2$ and $L(y)=\{b,\phi_1(z),\phi_2(z)\}$, in which case \aa{$\Phi_T(P',\C_0)$} contains a democracy with colors $\{\phi_1(z),\phi_2(z)\}$,
\aa{and $T$ contains a harmonica from $P$ to $P'$ with government $\C_0$}. 
To see this, note that for every color $d\in L(y)- \{b,\phi_1(z),\phi_2(z)\}$, $\Phi_T(P',\C_0)$ contains a dictatorship with dictator $\aa{y}$ in color $d$. Thus if $\Phi_T(P',\C_0)$ does not contain a confederacy, there exists at most one such color as otherwise it contains two dictatorships with dictator $\aa{y}$ in two different colors. But if there exists only one such color, then there exist $d'\in L(y)- \{b,d\}$ since $|L(y)|\ge 3$. But then there exists $d''\in \{\phi_1(z),\phi_2(z)\} - \{d'\}$ and hence $\Phi_T(P',\C)$ contains a dictatorship with dictator $z$ in color $d''$, a contradiction since then $\Phi_T(P',\C_0)$ contains a confederacy, a contradiction. Finally note that $|\C_0|=2$ as otherwise there exists $\phi_3\in \C$ and the same arguments as above imply that $L(y)=\{b,\phi_1(z),\phi_3(z)\}$, a contradiction.

Thus if $\C$ is a government, the arguments above imply that $\Phi_T(P',\C)$ contains a government and furthermore that $\Phi_T(P',\C)$ contains a confederacy unless $T$ contains a harmonica from $P$ to $P'$ with government $\C$,
contradicting that $T$ is a counterexample.

So suppose that $\C=\C_1\cup \C_2$ is a confederacy where $\C_1,\C_2$ are governments. From above, there exist governments $\C_1'\subseteq \Phi_T(P',\C_1)$ and $\C_2'\subseteq \Phi_T(P',\C_2)$. Since $C_1'\cup \C_2' \subseteq \Phi_T(P',\C)$, it follows that $\C_1'$ and $\C_2'$ are either both democracies in the same colors, or they are both dictatorship with the same dictator in the same color. But then from the arguments above, it follows that the same is true of $\C_1$ and $\C_2$ and hence $\C$ is not a confederacy, a contradiction.

\qed\end{proof}

\begin{claim}\label{NoIntersect}
$V(P)\cap V(P')=\emptyset$.
\end{claim}
\begin{proof}
Suppose not.
By Claim~\ref{NotEqual} $P\ne P'$. Let $\{z\}=V(P)\cap V(P')$ and let $P=xz$ and $P'=yz$. By Claim~\ref{NotTriangle}, $x$ is not adjacent to $y$. That is to say that $P\cup P'$ is an induced path of length two and yet there does not exist a chord of $C$ by Claim~\ref{HarmNoChord}. By Lemma~\ref{WheelUniqueColor}, there exists a unique $L$-coloring $\phi_0$ of $P\cup P'$ that does not extend to an $L$-coloring of $G$.

Suppose there do not exist $\phi_1,\phi_2\in \C$ such that $\phi_1(z)\ne \phi_2(z)$. But then $\C$ is a dictatorship with dictator $z$, say in color $c$. Let $\psi_1(z)=\psi_2(z)=c$ and let $\{\psi_1(y),\psi_2(y)\}$ be a subset of $L(y)- \{c\}$ of size two. Thus $\C'=\{\psi_1,\psi_2\}$ is a dictatorship on $P'$ with dictator $z$ in color $c$. Moreover, there exists $\phi\in \C$ such that $\phi(x)\ne \phi_0(x)$. Thus we can extend $\psi_1$ and $\psi_2$ to $L$-colorings of $G$ by letting $\psi_1(x)=\psi_2(x)=\phi(x)$. Hence $\C'\subseteq \Phi_T(P',\C)$, \aa{and $(P\cup P', P,L)$ is a harmonica from $P$ to $P'$ with government $\C$,} a contradiction.

So we may assume that there exist $\phi_1,\phi_2\in \C$ such that $\phi_1(z)\ne \phi_2(z)$. Let $i\in \{1,2\}$ be
such that $\phi_i(z)\ne \phi_0(z)$. Hence there is a dictatorship $\C_1\subseteq \Phi(P',C)$ such that $\phi(z)=\phi_i(z)$ for all $\phi\in \C_1$.

Suppose $L(y)- \{\phi_1(z),\phi_2(z),\phi_0(y)\} \ne \emptyset$. Let $c$ be a color in $L(y)- \{\phi_1(z),\phi_2(z),\phi_0(y)\}$. Hence there exists a dictatorship $\C_2\subseteq \Phi(P',\C)$ such that $\phi(y)=c$ for all $\phi\in \C_2$. But then $\Phi(P',\C)$ contains the confederacy $\C_1\cup \C_2$, a contradiction.

So we may assume that $L(y)=\{\phi_0(y),\phi_1(z),\phi_2(z)\}$ as $|L(y)|\ge 3$. Hence, $\phi_0(y)\ne \phi_1(z),\phi_2(z)$. Thus the democracy $\C_3$ in colors $\phi_1(z),\phi_2(z)$ is in $\Phi(P',\C)$. But then $\Phi(P',\C)$ contains the confederacy $\C_1\cup \C_3$, a contradiction.
\qed\end{proof}

\begin{claim}
$\C$ is a government.
\end{claim}
\begin{proof}
Suppose not. Then $\C=\C_1\cup \C_2$ is a confederacy. 
Note that by Claim~\ref{NotEqual}, $P\ne P'$, and by Claim~\ref{NotTriangle}, $V(P)\cup V(P')$ does not induce a triangle.
\REM{Deleted: So we may suppose that $V(P)\cup V(P')$ does not induce a triangle.}
By the minimality of $T$, since $\Phi_T(P', \C_1)$ does not contain a confederacy,
there exists a harmonica $T'=(G',P,L)$ from $P$ to $P'$ with government  $\C_1$. Since there does not exist a chord of $C$ by Claim~\ref{HarmNoChord} and $V(P)\cup V(P')$ does not induce a triangle, we find that $G'=P\cup P'$,
contrary to Claim~\ref{NoIntersect}.
\qed\end{proof}

\begin{claim}
$\C$ is a dictatorship.
\end{claim}
\begin{proof}
Suppose not. Hence $\C$ is a democracy. Let $L_0$ be the colors of $\C$. Let $Q=q_1\ldots q_k$ be a maximal path in $C$ such that $E(Q)\cap E(P')=\emptyset$, $V(P)\subseteq V(Q)$, $L_0\subseteq L(v)$ for all $v\in V(Q)$. Note that $C$ is a cycle since $G$ is $2$-connected by Claim~\ref{Harm2Conn}.

We claim that $q_1\in V(P')$. Suppose not. Let $q_1v_1\in E(C)- E(Q)$. Let $Q_1$ be the shortest subpath of $Q$ that includes $q_1$ and both vertices of $P$. Let $T'=(G\setminus V(Q_1), v_1\REM{Deleted: +P'},L')$ be the democratic reduction
$T(Q_1,L_0,v_1)$ of $Q_1$ with democracy $L_0$ centered around $v_1$. Note that $P'\subseteq G\setminus V(Q_1)$ since $V(P)\cap V(P')=\emptyset$ by Claim~\ref{NoIntersect}. As $Q$ is maximal and $v_1\not \in V(Q)$, $L_0$ is not a subset of $L(v_1)$ as otherwise $Q+v_1$ would also be a path satisfying the above conditions, contradicting that $Q$ is maximal. Hence $|L'(v_1)| = |L(v_1)|-|L(v_1)\cap L_0| \ge 3-1=2$. Let $P''$ be a path of length one of the boundary of the outer face of $G\setminus V(Q_1)$ containing $v_1$\aa{, and 
let $T''=(G\setminus V(Q_1),P'',L')$}.
Now the set of $L'$-colorings of $P''$ contains a confederacy $\C'$. By the minimality of $T$, $\Phi_{\aa{T''}}(P',\C')$ contains a confederacy. By Lemma~\ref{DemRed}(2), every $L'$-coloring of $G\setminus V(Q_1)$ extends to an $L$-coloring of $G$. Thus every coloring in $\Phi_{T''}(P',\C')$ is in $\Phi_T(P',\C)$. Hence $\Phi_T(P',\C)$ contains a confederacy, a contradiction. This proves the claim that $q_1\in V(P')$. By symmetry, it follows that $q_k\in V(P')$.

Yet $V(P)\cap V(P')=\emptyset$ by Claim~\ref{NoIntersect}. So $q_1,q_k\not\in V(P)$. Let $c_1\in L(q_1)- L_0$ and $c_2\in L(q_k)- L_0$. Let $\C_1=\{\phi_1,\phi_1'\}$ where $\phi_1(q_1)=\phi_1'(q_1)=c_1$ and $\{\phi_1(q_k),\phi_1'(q_k)\}=L_0$. Similarly, let $\C_2=\{\phi_2,\phi_2'\}$ where $\phi_2(q_k)=\phi_2'(q_k)=c_2$ and $\{\phi_2(q_1),\phi_2'(q_1)\}=L_0$. Hence $\C_1$ and $\C_2$ are distinct governments of $P'$ and $\C'=\C_1\cup \C_2$ is a confederacy. 

Moreover, for all $\phi\in \C'$, $\phi\in \Phi_T(P',\C)$. To see this, consider the democratic reductions $T_1=T(Q\setminus q_1, L_0, q_1)$ and $T_2=T(Q\setminus q_k, L_0, q_k)$. By Theorem~\ref{ThomPath2}, every $L$-coloring of $q_1$ extends to an $L$-coloring of $T_1$ and every $L$-coloring of $q_k$ extends to an $L$-coloring of $T_2$. Thus if $\phi \in \C_1$, then $\phi$ extends to an $L$-coloring of $T_1$ as noted above and can then be extended to $Q\setminus q_1$ by using the colors of $L_0$. Hence $\phi\in \Phi_T(P',\C)$. Similarly if $\phi\in \C_2$, then $\phi$ extends to an $L$-coloring of $T_2$ as noted above and can the be extended to $Q\setminus q_k$ by using the colors of $L_0$. Hence $\phi\in \Phi_T(P',\C)$. Thus $\Phi_T(P',\C)$ contains a confederacy, a contradiction.
\qed
\end{proof}

Suppose without loss of generality that $p_1$ is the dictator of $\C$ in color $c$. 
Let $v_1,v_2$ be the vertices of $C$ adjacent to $p_1$. Suppose without loss of generality that $P=p_1v_1$.
Our next objective is to consider two democratic reductions centered at $p_1$,
one removing $v_1$ and the other removing $v_2$. However, for the one removing $v_2$ we
need to define a new canvas so as to ensure that $v_1$ has a proper list. Moreover, we
first have to show that $v_1, v_2 \not\in V(P')$ if we are to study the colorings that extend to $P'$ in
these reductions.
 Let $\phi_1,\phi_2\in\C$. \aa{Then} $c=\phi_1(p_1)$. Let $M_1=\{\phi_1(v_1),\phi_2(v_1)\}$. Let $L'(v_1)=M_1\cup\{c\}$, $L'(p_1)=\{c\}$ and let $L'(w)=L(w)$ otherwise. Let $T'=(G,p_1,L')$. Let $M_2$ be a subset of $L(v_2)-\{c\}$ of size two.

\begin{claim}
\label{v1v2notinP'}
Neither $v_1$ nor $v_2$ is in $P'$.
\end{claim}
\begin{proof}
Suppose not. By Claim~\ref{NoIntersect} \aa{it follows} that $v_2$ is in $P'$. Let $P'=v_2z$. Let $S'=p_1v_2z$ and let $T''=(G,S',L')$. Since there does not exist a chord of $C$ by Claim~\ref{HarmNoChord}, $S'$ is an induced path of length two. Thus by Lemma~\ref{WheelUniqueColor}, there exists at most one proper $L'$-coloring of $S'$, call it $\phi_0$, that does not extend to an $L'$-coloring of $G$. 

Let $c_1\in L(v_2)- \{c,\phi_0(v_2)\}$. Let $\C_1'=\{\psi_1,\psi_1'\}$ where $\psi_1(v_2)=\psi_1'(v_2)=c_1$ and $\{\psi_1(z),\psi_1'(z)\}$ is a subset of $L(z)-\{c_1\}$ of size two. Thus $\C_1'$ is a dictatorship on $P'$ with dictator $v_2$. Suppose that $M_2=\{c_2,c_3\}$. 
If there exists $c_4\in L(z)- (\{\phi_0(z)\} \cup \aa{M_2})$, then let $\C_2'=\{\psi_2,\psi_2'\}$ where $\psi_2(z)=\psi_2'(z)=c_4$ and $\{\psi_2(v_2),\psi_2'(v_2)\}=M_2$. Otherwise $L(z)=M_2\cup\{\phi_0(z)\}$. In that case, let $\C_2'=\{\psi_3,\psi_3'\}$ where $\psi_3(v_2)=\psi_3'(z)=c_2$ and $\psi_3'(v_2)=\psi_3(z)=c_3$. Thus in the first case, $\C_2'$ is a dictatorship with dictator $z$ and in the second case $\C_2'$ is a democracy. Thus in either case, $\C'=\C_1'\cup \C_2'$ is a confederacy. Yet in either case, $\C' \aa{\subseteq \Phi_{T'}(P',\C)}\subseteq \Phi_T(P',\C)$, since in every coloring in $\C'$ either $v_2$ or $z$ receives a color different from the color it receives in $\phi_0$, a contradiction.
\qed
\end{proof}

Now consider the democratic reductions $T_1=T'(v_1, M_1, p_1), T_2=T'(v_2, M_2, p_1)$. Suppose that $T_1=(G_1,p_1,L_1)$ and $T_2=(G_2,p_1,L_2)$. Let $T_1'=(G_1, p_1v_2, L_1)$ and $T_2'=(G_2,P,L_2)$. By the minimality of $T$, $\Phi_{T_2'}(P',\C)$ contains a government $\C_2$. Furthermore, either $\Phi_{T_2'}(P',\C)$ contains a confederacy, a contradiction as then so does $\Phi_T(P',\C)$, or, $T_2'$ contains a harmonica $T_2''=(G_2',P,L_2)$ from $P$ to $P'$ with government $\C$. 

Let $\C^*=\{\phi_1,\phi_2,\ldots \phi_k\}$ where $\phi_1(p_1)=\phi_2(p_1)=\ldots=\phi_k(p_1)=c$ and $\{\phi_1(v_2),\phi_2(v_2),\ldots,\allowbreak\phi_k(v_2)\} = L(v_2)- \{c\}$. By the minimality of $T$, where we consider the canvas $T_1'$, we find that $\Phi_{T_1'}(P',\C^*)$ contains a government $\C_1$. Furthermore, either $\Phi_{T_1'}(P',\C^*)$ contains a confederacy, a contradiction as then so does $\Phi_T(P',\C)$, or, $T_1'$ contains a harmonica $T_1''=(G_1',p_1v_2,L_1)$ from $P$ to $P'$ with government $\C^*$.

Let $P'=p_1'p_2'$ where $p_1'$ is on the subpath of $C$ from $p_2'$ to $v_1$ not containing $p_1$.

\begin{claim}
There exists $v\not\in C$ such that $v\sim p_1,v_1,v_2,p_1',p_2'$
\end{claim}
\begin{proof}
Let $W_1$ be the outer walk of $G_1'$ and $W_2$ be the outer walk of $G_2'$. 
\xx{Since $T_1''$ is obtained by an application of the third rule by Claim~\ref{v1v2notinP'}, the vertex $p_1$
has two neighbors $u_1,u_2$ such that $|L_1(u_i)|=3$ if $u_i\ne v_2$. But not both $u_1,u_2$ can
be adjacent to $v_1$ by planarity and Claim~\ref{HarmInternal}, and hence one of them, say $u_2$ belongs to $C$. It follows from Claim~\ref{HarmNoChord} 
that $u_2=v_2$. Thus $v_2\in W_1$ and $p_1$ is not a cutvertex of $G_1'$.}
Let $W_1'$ be the subwalk of $W_1$ from $p_1$ to $P'$ not containing $v_2$. Let $W_1''$ be the subwalk of $W_1$ from $p_1$ to $P'$  containing $v_2$. 

Note if $z\in V(G_1')-V(C)$, then $|\aa{L_1}(z)|\le 3$ \xx{by Lemma~\ref{L=3}}. Hence  $|L(z)|=5$, $|\aa{L_1}(z)|=3$, $M_1\subseteq L(z)$ and $z$ is adjacent to $v_1$. However if $z\in V(W_1'')- V(W_1')$, then $z$ is not adjacent to $v_1$ since $G$ is planar and hence $z\in V(C)$. 

We claim that there exists $w_1\in V(G_1')-V(C)$ such that $w_1$ is adjacent to $v_1$ and $p_2'$. To see this, note that there exists a vertex $w_1\in V(W_1')$ in a triangle $R$ in $G_1'$ such that either $R=w_1p_1'p_2'$ or $R=w_1p_i'u_1$ for some $i\in\{1,2\}$ where $u_1$ in $V(W_1'')- V(W_1')$. In the former case, $w_1\not\in V(C)$ by Claim~\ref{HarmNoChord} and hence $w_1$ is adjacent to $v_1$ as desired. 
In the latter case, it follows that that $u_1$ is not adjacent to $v_1$, and hence $u_1\in V(C)$
\xx{by the result of the previous paragraph}. But then
$i=2$ by Claim~\ref{HarmNoChord}. Moreover,
by Claim~\ref{HarmNoChord}, $w_1\not\in V(C)$ given that $w_1$ is adjacent to $u_1$ and hence $w_1$ is adjacent to $v_1$ as desired.

By symmetry, there exists $w_2\in V(G_2')$ such that $w_2$ is adjacent to $v_2$ and $p_1'$. As $G$ is planar, we find that $w_1=w_2$, call it $v$, and hence $v$ is adjacent to $u_1,u_2,v_1,v_2$. Moreover as $T_2''$ is a harmonica with government $\C$ and $\C$ is a dictatorship with dictator $p_1$, we find that $v$ is adjacent to $p_1$.
\qed\end{proof}

%

\begin{claim}
\label{cl:z1zk}
There exist vertices $v_2=z_1,z_2,...,z_k=p_2'$, all adjacent to $v$ such that for all $i=1,2,\ldots,k$ the canvas
$(G_1'\setminus \{p_1,z_1,...,z_{i-1}\},vz_i,L_1)$ is a harmonica from $vz_i$ to $P'$ with a democracy or a dictatorship depending on the parity of $i$. 
\end{claim}
\begin{proof}
The canvas $(G_1'\setminus p_1,vz_1,L_1)$ is a harmonica from $vz_1$ to $P'$ with a democracy obtained from $T_1''$ by
application of the third rule, $(G_1'\setminus \{p_1,v_2\},vz_2,L_1)$ is a harmonica from $vz_2$ to $P'$ with a dictatorship obtained from $(G_1'\setminus p_1,vz_1,L_1)$
by  application of the fourth rule, and so on. We note that the vertex $v$ cannot be the vertex that is being deleted during the construction of the next harmonica, 
\xx{because the next-to-last harmonica in the sequence leading up to $P'$ involves the vertex $v$.}
%
This proves Claim~\ref{cl:z1zk}.
\qed\end{proof}

We are now ready to complete the proof of  Theorem~\ref{Harmonica2}.
It follows that  $M_2$ is a subset of \xx{$L(v)$ and} $L(z_i)$ for all $i, 1\le i \le k$. Similarly there exist vertices $v_1=w_1,...,w_l=p_1'$ and $M_1$ is a subset of \xx{$L(v)$ and} $L(w_i)$ for all $i, 1\le i \le l$. Since $M_1\cup M_2\cup\{c\}=L(v)$  \xx{by Lemma~\ref{L=3}}, we see that $M_1$ and $M_2$ are disjoint. 
Since $|L(p_1')|=|L(p_2')|=3$ \xx{by Lemma~\ref{L=3}} as $T_1''$ and $T_2''$ are harmonicas, and $M_1\subseteq L(p_1')$ and $M_2\subseteq L(p_2')$, it follows that the last step in the construction of $T_1''$ is according to the second rule of the definition of harmonica. In other words, $p_2'$ is a dictator of $\C_1$. Similarly, $p_1'$ is a dictator of $\C_2$. Thus $\C_1\cup\C_2$ is a confederacy, as desired.
This completes the proof of Theorem~\ref{Harmonica2}.~\qed
\end{proof}

In the following lemma we 
clarify the relationship between harmonicas and coloring harmonicas before we can prove
Theorem~\ref{thm:harmon}.

\begin{lem}
\label{lem:harmonicas}
Let $T=(G,P,L)$ be a canvas and $P,P'$ be  paths of length one in the boundary of the outer face of $G$. 
Let $P=uv$, let $P'=ww'$, where $u,v,w$ are pairwise distinct, and let $T$ be a harmonica from $P$ to $P'$ with government $\C$.
Assume that $\Phi(P',\C)$ is a dictatorship with dictator $w$ in color $d$, that if $\C$ is a democracy, then $|L(u)|=|L(v)|=2$, and 
that if $\C$ is a dictatorship, then $u$ is the dictator, $|L(u)|=1$ and \xx{if $v$ has degree in $G$ of at least two, then} 
$|L(v)-L(u)|=2$.
Let $G'$ be obtained from $G$ by deleting either or both of $v$ and $w'$ if either has degree one in $G$,
let $L'(w)=L(w)-\{d\}$ and  $L'(x)=L(x)$ for all $x\in V(G')-\{w\}$.
If $\C$ is a dictatorship, then $(G',L')$ is a coloring harmonica from $u$ to $w$.
If $\C$ is a democracy, then $(G',L')$ is a coloring harmonica from $uv$ to $w$.
\end{lem}

\begin{proof}
We proceed by induction on $|V(G)|$.
If $|V(G)|=3$, then, since $w$ is the dictator of $\Phi(P',\C)$, and $w\not\in V(P)$, we deduce that $T$ is obtained according to the 
fourth rule in the definition of harmonica. Thus $\C$ is a democracy, $u,v,w$ form a triangle, $L'(u)=L'(v)=L'(w)$ and $|L(u)|=2$.
Thus $(G',L')$ is a coloring harmonica from $uv$ to $w$.
We may therefore assume that  $|V(G)|\ge4$.

Assume now that $\C$ is a democracy. Thus $T$ is obtained according to the 
fourth rule in the definition of harmonica. Consequently, there exists a vertex $z$ adjacent to $u$ and $v$
such that $L(z)=L_0\cup \{c\}$ where $L_0=\C(u)=\C(v)$ and there exist $x,y$ such that $\{x,y\}=\{u,v\}$,
 the canvas $T'=(G\setminus x, U, L)$ is a harmonica with dictatorship $\C'=\{\phi_1,\phi_2\}$, where $U=zy$ and 
$\phi_1(z)=\phi_2(z)=c$ and $\{\phi_1(y),\phi_2(y)\}=L_0$.
It follows that $L(u)=L(v)=L_0$.
If $z=w$, then $T'$ is obtained according to the second rule in the definition of harmonica, and hence $w'$ has degree 
one in $G$. It follows that $(G',L')$ is a coloring harmonica from $uv$ to $w$, because $G'$ is obtained from $G$
by deleting $w'$.
We may therefore assume that $z\ne w$.

Thus $T'$ is obtained according to the third rule in the definition of harmonica. Let $u_1,u_2$ be as in that rule,
and let $G_1'=G'\setminus x$ if $y\in\{u_1,u_2\}$ and $G_1'=G'\setminus \{x,y\}$ otherwise. 
Let $L_1(z)=L(z)-L_0$ and $L_1(z')=L(z')$ for all $z'\in V(G)-\{z\}$,
and let 
$L_1'(x')=L'(x')$ for all $x'\in V(G_1')$.
It follows by induction applied to the canvas $(G\setminus x, U, L_1)$ that $(G_1',L_1')$ is a coloring harmonica from $z$ to $w$, 
and hence $(G',L')$ is a coloring harmonica from $uv$ to $w$, as desired.

We may therefore assume that $\C$ is a dictatorship. Since $w\not\in V(P)$, it follows that $T$ is obtained according to the 
 third rule in the definition of harmonica. Thus  there exists a triangle $uu_1u_2$ and letting $c$ denote the color in which $u$
is the dictator,
for $i=1,2$ we have $L(u_i)=L_0\cup\{c\}$ if $u_i\not\in V(P)$ and $\C(u_i)=L_0$ otherwise,
where $|L_0|=2$ and the canvas $(G\setminus (V(P)- V(U)), U,L)$ is a harmonica from $U=u_1u_2$ to $P'$ with democracy $\C'$, where $\C'(u_1)=\C'(u_2)=L_0$. Let us note that $|L(u)|=1$ and if $u_i\in V(P)$, then $u_i=v$, and hence $|L(u_i)-L(u)|=2$.
Let $G_1'=G'\setminus u$, let $L_1(u_1)=L_1'(u_1)=L(u_1)-L(u)$, $L_1(u_2)=L_1'(u_2)=L(u_2)-L(u)$, and let $L_1(u')=L(u')$ and 
$L_1'(u')=L'(u')$  for all $u'\in V(G_1')-\{u_1,u_2\}$.
It follows by induction applied to the canvas $(G\setminus (V(P)- V(U)), U,L_1)$ that $(G_1',L_1')$ is a coloring harmonica from $u_1u_2$ to $w$, 
and hence $(G',L')$ is a coloring harmonica from $u$ to $w$, as desired.
\qed
\end{proof}

\bigskip
\noindent{\bf Proof of Theorem~\ref{thm:harmon}.}
We prove only the backward direction as the forward direction is trivial. So let $G,p_1,p_2,L$ be as in the statement of the theorem and suppose for the sake of contradiction that $G$ is not $L$-colorable and yet $(G,L)$ does not contain a coloring harmonica from $p_1$ to $p_2$. 

Let $p_1v_1,p_2v_2\in E(C)$ and let $P=p_1v_1$ and $P'=p_2v_2$. Let $c\in L(p_1)$. 
Let $\C_c=\{\phi_1,\phi_2, \ldots \phi_k\}$ where $\phi_i(p_1)=c$ for all $i, 1\le i \le k$, and $\{\phi_1(v_1),\phi_2(v_2),\ldots, \phi_k(v_1)\} = L(v_1)- \{c\}$. 
Note that $\C_c$ is a dictatorship on $P$ with dictator $p_1$ in color $c$. 
\aa{If there exists $c'\in L(p_1)-\{c\}$, then let $\C=\C_c\cup\C_{c'}$, in which case $\C$ is a confederacy;
otherwise let $\C=\C_c$, in which case  $\C$ is a dictatorship on $P$ with dictator $p_1$ in color $c$. }
Let $c_0$ be a new color and let $L'(p_2)=L(p_2)\cup \{c_0\}$\aa{, and let $L'(x)=L(x)$ for all $x\in V(G)-\{p_2\}$}. 
Let $T=(G,P,L')$. By Theorem~\ref{Harmonica2}, $\Phi_T(P',\C)$ contains a government $\C'$. Furthermore, either $\Phi_T(P',\C)$ contains a confederacy $\C''$, or,
\aa{$\C$ is a government and}
 there exists a harmonica $T'=(G',P,L')$ from $P$ to $P'$ with government $\C$. 

In the former case, there exists a coloring $\phi\in \C''$ such that $\phi(p_2)\ne c_0$. But then $\phi$ extends to an $L'$-coloring of $G$ as $\phi \in \Phi_T(P',\C)$ and yet $\phi$ is an $L$-coloring of $G$ since $\phi(p_2)\ne c_0$, a contradiction. 
So we may assume the latter case. \aa{Thus $|L(p_1)|=1$, because $\C$ is a government; and
\xx{if $v_1$ has degree at least two in $G$, then}  $k=2$, because $T'$
is a harmonica obtained according to the third rule.} As above, there does not exist a coloring $\phi\in \C'$ such that $\phi(p_2)\ne c_0$. Hence $\C'$ is a dictatorship with dictator $p_2$ in color $c_0$. Let $G''$ be obtained from $G'$ by deleting either or both of $v_1$ and $v_2$ if either has degree one in $G'$. Now $(G'',L)$ is a coloring harmonica from $p_1$ to $p_2$  \aa{by Lemma~\ref{lem:harmonicas}}, a contradiction.
\qed

%

\xx{%
\section*{Acknowledgment}
The results of this paper form part of the doctoral dissertation~\cite{PosPhD} of the first author,
written under the guidance of the second author.}

\baselineskip 11pt
\vfill
\noindent
This material is based upon work supported by the National Science Foundation.
Any opinions, findings, and conclusions or
recommendations expressed in this material are those of the authors and do
not necessarily reflect the views of the National Science Foundation.
\eject

\end{document}